\documentclass[12pt]{amsart}

\input xy
\usepackage[english]{babel}
\usepackage[latin1]{inputenc}
\usepackage{graphicx}
\usepackage{amsmath}
\usepackage{amssymb}
\usepackage{amscd}
\usepackage{color}
\usepackage{oldgerm}
\usepackage{amsfonts}
\usepackage{newlfont}
\usepackage{longtable}
\usepackage{multirow}
\usepackage{xcolor}

\DeclareOldFontCommand{\rm}{\normalfont\rmfamily}{\mathrm}

\usepackage[all]{xy}

\usepackage{upref}

\def\F{\Bbb F}

\def\d{\operatorname{d}}

\def\Ker{\operatorname{Ker}}

\def\Id{\operatorname{Id}}

\def\Im{\operatorname{Im}}

\def\D{\operatorname{D}}

\def\g{\frak g}
\def\gl{\frak{gl}}

\theoremstyle{plain}\swapnumbers

\newtheorem{Theorem}{Theorem}[section]

\newtheorem{Prop}[Theorem]{Proposition}

\newtheorem{Remark}[Theorem]{Remark}

\newtheorem{claim}{Claim}

\title[On Cohomology group of current Lie algebras]{On Cohomology group of current Lie algebras}

\author{R. Garc\'{\i}a-Delgado}
\address{Centro de Investigaci\'on en Matem\'aticas A. C.,  Unidad M\'erida; Yucat\'an, M\'exico, Carretera Sierra Papacal Chuburna Puerto Km 5, 97302 Sierra Papacal, Yuc.}
\email{rosendo.garciadelgado@alumnos.uaslp.edu.mx}



\date{\today}

\begin{document}

\maketitle

\begin{abstract}
In this work we state a result that relates the cohomology groups of a Lie algebra $\g$ and a current Lie algebra $\g \otimes \mathcal{S}$, by means of a short exact sequence -- similar to the universal coefficients theorem for modules --, where $\mathcal{S}$ is a finite dimensional, commutative and associative algebra with unit over a field $\F$. Although this result can be applied to any Lie algebra, we determine the cohomology group of $\g \otimes \mathcal{S}$, where $\g$ is a semisimple Lie algebra. 
\end{abstract}

\section*{Introduction}

Let $\g$ be a Lie algebra with bracket $[\,\cdot\,,\,\cdot\,]$ and $\mathcal{S}$ be an associative and commutative algebra over $\F$ with product $(s,t) \mapsto s\,t$, for all $s,t$ in $\mathcal{S}$. The skew-symmetric and bilinear map $[\,\cdot\,,\,\cdot\,]_{\g \otimes \mathcal{S}}$ defined on $\g \otimes \mathcal{S}$, by:
$$
[x \otimes s,y \otimes t]_{\g \otimes \mathcal{S}}=[x,y] \otimes s\,t,\,\, \text{ for all }\,x,y \in \g,\,\text{ and }\,s,t \in \mathcal{S},
$$ 
yields a Lie algebra in $\g \otimes \mathcal{S}$, which is called the {\bf current Lie algebra of $\g$ by $\mathcal{S}$}.
\smallskip

Let $\rho:\g \to \gl(V)$ be a representation of $\g$ on a vector space $V$, then $V$ is called a \textbf{$\g$-module}. The representation $\rho$ can be extended to a representation $R$ of $\g \otimes \mathcal{S}$ on the vector space $V \otimes \mathcal{S}$ by means of:
\begin{equation}\label{representacion current}
R(x \otimes s)(v \otimes t)=\rho(x)(v) \otimes s\,t,\text{ for all }x,y \in \g,\,v \in V,\,s,t \in \mathcal{S}.
\end{equation}

Let $C(\g;V)=C^0(\g;V) \oplus \ldots \oplus C^{p}(\g;V) \oplus \ldots$ be the \textbf{space of cochains from $\g$ into $V$}, where $C^0(\g,V)$ is equal to $V$ and $C^p(\g;V)$ is the space of the alternating $p$-multilinear maps of $\g$ with values in $V$. For any $\g$-mdule $V$, let $\d:C(\g;V) \to C(\g;V)$ be the differential map given by:
\begin{equation}\label{Dif}
\begin{split}
& \d\lambda(x_1,\ldots,x_{p+1})=\sum_{j=1}^{p+1}(-1)^{j-1}\rho(x_j)(\lambda(x_1,\ldots,x_{\hat{j}},\ldots,x_{p+1}))\\
\,&+\sum_{j<k}(-1)^{j+k}\lambda([x_j,x_k],x_1,\ldots,x_{\hat{j}},\ldots,x_{\hat{k}},\ldots,x_{p+1}),\qquad p>0.
\end{split}
\end{equation} 
where $\lambda$ is in $C^p(\g;V)$ and $x_1,\ldots,x_{p+1}$ are in $\g$. For $p=0$ we let $\d(v)(x)=\rho(x)(v)$ where $v$ is in $V$ and $x$ is in $\g$. 
\smallskip

The aim of this work is to set a result that relates the cohomology groups $\mathcal{H}(\g \otimes \mathcal{S};V \otimes \mathcal{S})$ and $\mathcal{H}(\g;V)$, similar to the Universal coefficient theorems for modules (see \cite{Cartan}, Chapter VI, \S 3, \textbf{Thm. 3.3}). 
\smallskip

To achieve our goal, in \textbf{Prop. \ref{composicion}} we introduce a map between the set of cochains of $\g$ and cochains of $\g \otimes \mathcal{S}$. Next we prove that there exists a surjective linear map $\alpha$ between $\mathcal{H}(\g \otimes \mathcal{S};V \otimes \mathcal{S})$ and $\mathcal{H}(\g;V) \otimes \mathcal{S}$ (see \textbf{Prop. \ref{morfismo natural}}). In \textbf{Thm. \ref{teorema de coeficientes universales}} we determine the kernel of $\alpha$ and we state a result that relates the cohomology groups $\mathcal{H}(\g \otimes \mathcal{S};V \otimes \mathcal{S})$ and $\mathcal{H}(\g;V) \otimes \mathcal{S}$ by means of a short exact sequence. Finally, we determine the cohomology group $\mathcal{H}(\g \otimes \mathcal{S};V \otimes \mathcal{S})$, when $\g$ is a semisimple Lie algebra and $V$ is an irreducible $\g$-module. 
\smallskip

The results obtained here are focused at knowing the cohomology group $\mathcal{H}(\g \otimes \mathcal{S};V \otimes \mathcal{S})$, based on the cohomology group $\mathcal{H}(\g;V)$. In \cite{Zusmanovich-1}, is determined formulas for the first and second cohomology groups of a current Lie algebra $\g \otimes \mathcal{S}$, with coefficients in a module $V \otimes \mathcal{A}$, where $V$ is a $\g$-module and $\mathcal{A}$ is a $\mathcal{S}$-module. In \textbf{Thm. 2.1} of \cite{Zusmanovich}, is given a formula for the second cohomology group of $\g \otimes \mathcal{S}$, with coefficients in the trivial module and $\mathcal{S}$ has no unit. In  \cite{Neeb} is given a description of the cohomology group $\mathcal{H}(\g \otimes \mathcal{S};\mathcal{V})$, where $\mathcal{V}$ is a trivial module of $\g \otimes \mathcal{S}$. It seems that one of the first results with this focus appears in \cite{Berezin}, where it is shown that cohomology of $\g \otimes \mathcal{S}$, where $\mathcal{S}$ is a local algebra, can be reduced to cohomology of $\g$. On the other hand, it is unknown if there exists a criterion for recognizing whether an arbitrary Lie algebra is a current Lie algebra. A step in this direction can be found in \cite{Chaktoura}, where some examples are given in 4-dimensional current Lie algebras.

\section{The map $\mathcal{L}:C(\g \otimes \mathcal{S};V) \to C(\g;V)$}

The proof of the following result is standard and we omit it.

\begin{Prop}\label{prop1}{\sl
Let $V$ and $\mathcal{S}$ be finite dimensional vector spaces over $\F$. Let $\{s_1,\ldots,s_m\}$ be a basis of $\mathcal{S}$. For any $x$ in $V \otimes \mathcal{S}$ there are $v_1,\ldots,v_m$ in $V$ such that: $x=v_1 \otimes s_1+\ldots+v_m \otimes s_m$.
}
\end{Prop}

Let $\g \otimes \mathcal{S}$ be the current Lie algebra of $\g$ by $\mathcal{S}$, where $\mathcal{S}$ is a $m$-dimensional commutative and associative algebra with unit $1$ over $\F$. We use the same symbol for the bracket on $\g \otimes \mathcal{S}$ and the bracket on $\g$, that is: $[x \otimes s,y \otimes t]=[x,y] \otimes st$, for all $x,y$ in $\g$ and $s,t$ in $\mathcal{S}$. 
\smallskip

We fix a basis $\{s_1,\ldots,s_m\}$ of $\mathcal{S}$, where $s_1=1$. Let $\rho:\g \to \gl(V)$ be a representation of $\g$ on a finite dimensional vector space $V$. Let $\bar{x}_1,\ldots,\bar{x}_p$ be in $\g \otimes \mathcal{S}$ and $\Lambda$ in $C^p(\g \otimes \mathcal{S};V \otimes \mathcal{S})$; as $\Lambda(\bar{x}_1,\ldots,\bar{x}_p)$ lies in $V \otimes \mathcal{S}$, by {\bf Prop. \ref{prop1}} we write $\Lambda(\bar{x}_1,\ldots,\bar{x}_p)$, as follows:
\begin{equation}\label{j1}
\Lambda(\bar{x}_1,\ldots,\bar{x}_p)=\Lambda_1(\bar{x}_1,\ldots,\bar{x}_p) \otimes s_1+\ldots+\Lambda_m(\bar{x}_1,\ldots,\bar{x}_p) \otimes s_m,
\end{equation}
where $\Lambda_j(\bar{x}_1,\cdots,\bar{x}_p)$ belongs to $V$ for all $j$. As $\Lambda$ is in $C^p(\g \otimes \mathcal{S};V \otimes \mathcal{S})$, the map $(\bar{x}_1,\ldots,\bar{x}_p) \mapsto \Lambda_j(\bar{x}_1,\ldots,\bar{x}_p)$ belongs to $C^p(\g \otimes \mathcal{S};V)$. We denote this map by $\Lambda_j$ for all $1 \leq j \leq m$.
\smallskip

Let $\{\omega_1,\ldots,\omega_m\} \subset \mathcal{S}^{\ast}$ be the dual basis of $\{s_1,\ldots,s_m\}$. For each $j$, the bilinear map $(v,s) \mapsto \omega_j(s)v$ yields the linear map $\widehat{\omega}_j:V \otimes \mathcal{S} \to V$, given by $\widehat{\omega}_j(v \otimes s)=\omega_j(s)\,v$.
\smallskip

For each $\!j$, consider the map $\chi_j\!\!:\!C(\g \!\otimes \!\mathcal{S};\!V \!\otimes \!\mathcal{S}) \!\!\to \!C(\g \!\otimes \!\mathcal{S};\!V)$ defined by:
\begin{equation}\label{chi}
\begin{split}
& \chi_j(\Lambda)=\widehat{\omega}_j \circ \Lambda,  \quad \text{ for } \Lambda \in C^p(\g \otimes \mathcal{S};V \otimes \mathcal{S}),\,\,p>0,\,\text{ and }\\
& \chi_j(v \otimes s)=\widehat{\omega}_j(v \otimes s), \quad \text{ for } v \in V\,\text{ and }s \in \mathcal{S}.
\end{split}
\end{equation}
In particular, for a given $\Lambda$ in $C^p(\g \otimes \mathcal{S};V \otimes \mathcal{S})$, where $p>0$, it follows:
\begin{equation}\label{widehat}
\begin{split}
& \text{ \emph{If} }\Lambda(\bar{x}_1,\ldots,\bar{x}_p)\!=\!\Lambda_1(\bar{x}_1,\ldots,\bar{x}_p)\! \otimes \!s_1\!+\ldots +\!\Lambda_m(\bar{x}_1,\ldots,\bar{x}_p) \!\otimes \!s_m,\\
& \,\text{ \emph{then} }\,\Lambda_j=\chi_j(\Lambda)=\widehat{\omega}_j \circ \Lambda\,\text{ for each }1 \leq j \leq m.
\end{split}
\end{equation}
Let $\mathcal{L}:C(\g \otimes \mathcal{S};V) \to C(\g;V)$ be the map defined by:
\begin{equation}\label{j2}
\begin{split}
& \mathcal{L}(v)=v, \text{ for all }v \in V,\,\text{ and }\\
& \mathcal{L}(\lambda)(x_1,\ldots,x_p)=\lambda(x_1 \otimes 1,\ldots,x_p \otimes 1),
\end{split}
\end{equation}
where $\lambda$ is in $C^p(\g \otimes \mathcal{S};V)$, and $x_1,\ldots,x_p$ are in $\g$. Let $\D$ be the differential in $C(\g \otimes \mathcal{S};V \otimes \mathcal{S})$ (see \eqref{Dif}). In the next result we will prove that $\D$, $\d$, $\mathcal{L}$ and $\chi_j$ can be inserted into a commutative diagram.
\begin{Prop}\label{prop L y d conmutan}{\sl
The following diagram is commutative:
\begin{equation}\label{L y d conmutan}
\xymatrix{
C^p(\g \otimes \mathcal{S};V \otimes \mathcal{S}) \ar[r]^{\chi_j} \ar[d]_{\D} & C^p(\g \otimes \mathcal{S};V) \ar[r]^{\mathcal{L}} & C^p(\g;V) \ar[d]^{\d}\\
C^{p+1}(\g \otimes \mathcal{S};V \otimes \mathcal{S}) \ar[r]^{\chi_j} & C^{p+1}(\g \otimes \mathcal{S};V) \ar[r]^{\mathcal{L}} & C^{p+1}(\g;V)
}
\end{equation}
In particular, for $p>0$ and $\Lambda$ in $C^p(\g \otimes \mathcal{S};V \otimes \mathcal{S})$ as in \eqref{widehat}:
\begin{equation}\label{j3}
\mathcal{L}\left((\D\Lambda)_j\right)=\d\mathcal{L}(\Lambda_j),\,\text{ for all }1 \leq j \leq m.
\end{equation}
}
\end{Prop}
\begin{proof}
First let us consider $p>0$ and $\Lambda$ in $C^p(\g \otimes \mathcal{S};V \otimes \mathcal{S})$ as in \eqref{widehat}. Applying \eqref{widehat} to $\D \Lambda$ we obtain: $(\D \Lambda)_j=\chi_j(\D \Lambda)$, for all $j$. Then \eqref{j3} holds if and only the diagram \eqref{L y d conmutan} is commutative, that is:
\begin{equation}\label{aclaracion}
\begin{split}
& \mathcal{L}\left((\D \Lambda)_j\right)=\mathcal{L}\left(\chi_j \circ \D \Lambda  \right)=\mathcal{L}\circ \chi_j \circ \D \Lambda,\,\text{ and }\\
& \d \left(\mathcal{L}(\Lambda_j)\right)=\d \mathcal{L}(\chi_j \circ \Lambda)=\d \circ \mathcal{L} \circ \chi_j(\Lambda).
\end{split}
\end{equation}
We shall prove that $\mathcal{L}\left((\D \Lambda)_j\right)=\d \left(\mathcal{L}(\Lambda_j)\right)$, for all $j$. Let $X_i=x_i \otimes 1$, where $x_i$ belongs to $\g$, for all $1 \leq i \leq p+1$. By \eqref{Dif} we have:
\begin{equation}\label{dl-11}
\begin{split}
& \D \Lambda(X_1,\ldots,X_{p+1})\!=\!\sum_{i=1}^{p+1}(-1)^{i-1}\!R(X_i)\!\left(\Lambda(X_1,\ldots,X_{\bar{i}},\ldots,X_{p+1})\right)\\
&+\sum_{i<k}(-1)^{i+k}\Lambda([X_i,X_k],X_1,\ldots,X_{\hat{i}},\ldots,X_{\hat{k}},\ldots,X_{p+1}).
\end{split}
\end{equation}
Applying \eqref{widehat} to $\Lambda$ in \eqref{dl-11} above, it follows:
\begin{equation}\label{dL1}
\begin{split}
&\D\Lambda(X_1,\ldots,X_{p+1})\\
&=\!\sum_{i=1}^{p+1}\!\sum_{j=1}^m(-1)^{i-1}R(X_i)\left(\Lambda_j(X_1,\ldots,X_{\bar{i}},\ldots,X_{p+1})\! \otimes\! s_j\right)\\
&+\!\sum_{i<k}\!\sum_{j=1}^m(-1)^{i+k}\Lambda_j([X_i,X_k],X_1,\ldots,X_{\bar{i}},\ldots,X_{\bar{k}},\ldots,X_{p+1})\!\otimes \!s_j
\end{split}
\end{equation}
Let us analyze each of the terms $R(X_i)\left(\Lambda_j(X_1,\ldots,X_{\bar{i}},\ldots,X_{p+1})\! \otimes\! s_j\right)$ and $\Lambda_j([X_i,X_k],X_1,\ldots,X_{\bar{i}},\ldots,X_{\bar{k}},\ldots,X_{p+1})\!\otimes \!s_j$, given in \eqref{dL1}. Applying the representation $R$ (see \eqref{representacion current}) and $\mathcal{L}$ (see \eqref{j2}), we obtain:
\begin{equation}\label{b1}
\begin{split}
& R(X_i)\left(\Lambda_j(X_1,\ldots,X_{\hat{i}},\ldots,X_{p+1}) \otimes s_j\right)\\
&=\rho(x_i)\left(\Lambda_j(X_1,\ldots,X_{\hat{i}},\ldots,X_{p+1})\right) \otimes s_j\\
&=\rho(x_i)\left( \mathcal{L}(\Lambda_j)(x_1,\ldots,x_{\hat{i}},\ldots,x_{p+1})\right)\otimes s_j\,
\end{split}
\end{equation}
as $\Lambda_j(X_1,\ldots,X_{\hat{i}},\ldots,X_{p+1})\!\!=\!\!\mathcal{L}(\Lambda_j)(x_1,\ldots,x_{\hat{i}},\ldots,x_{p+1})$. In addition,
\begin{equation}\label{b1-2}
\begin{split}
& \Lambda_j([X_i,X_k],X_1,\ldots,X_{\hat{i}},\ldots,X_{\hat{k}},\ldots,X_{p+1})\\
&=\mathcal{L}(\Lambda_j)([x_i,x_k],\ldots,x_{\hat{i}},\ldots,x_{\hat{k}},\ldots,x_{p+1}),\,\text{ for all }1 \leq j \leq m,
\end{split}
\end{equation}
as $[X_i,X_k]\!=\![x_i,x_k]\!\otimes \!1$. We substitute \eqref{b1}-\eqref{b1-2} in \eqref{dL1}, and we get:
\begin{equation}\label{NEQ1}
\begin{split}
& \D \Lambda(X_1,\ldots,X_{p+1})=\D \Lambda(x_1 \otimes 1,\ldots,x_p \otimes 1)\\
&=\sum_{i=1}^{p+1}\sum_{j=1}^m(-1)^{i-1}\rho(x_i)\left(\mathcal{L}(\Lambda_j)(x_1,\ldots,x_{\hat{i}},\ldots,x_{p+1})\right)\otimes s_j\\
&+\sum_{i<k}\sum_{j=1}^m(-1)^{i+k}\mathcal{L}(\Lambda_j)([x_i,x_k],\ldots,x_{\hat{i}},\ldots,x_{\hat{k}},\ldots,x_{p+1})\otimes s_j.
\end{split}
\end{equation}
In \eqref{NEQ1} we gather the terms corresponding at each $s_j$, to obtain:
\begin{equation}\label{b2}
\D \Lambda(x_1 \otimes 1,\ldots,x_p \otimes 1)=\sum_{j=1}^m \d\left( \mathcal{L}(\Lambda_j)\right)(x_1,\ldots,x_{p+1}) \otimes s_j.
\end{equation}
On the other hand, applying \eqref{widehat} to $\D \Lambda$, we also obtain:
\begin{equation}\label{b3}
\D \Lambda(x_1 \otimes 1,\ldots,x_p \otimes 1)=\sum_{j=1}^m \left(\D(\Lambda)\right)_j(x_1 \otimes 1,\ldots,x_p \otimes 1) \otimes s_j
\end{equation}
By \eqref{j2}, $ (\D\Lambda)_j(x_1 \otimes 1,\ldots,x_p \otimes 1)\!=\!\mathcal{L}\left((\D \Lambda)_j\right)(x_1,\ldots,x_p)$. Thus, from \eqref{b2} and \eqref{b3} we get $\mathcal{L}\left((\D\Lambda)_j\right)=\d\mathcal{L}(\Lambda_j)$, for each $j$. Hence by \eqref{aclaracion}, it follows that for $p>0$, the diagram \eqref{L y d conmutan} is commutative. 
\smallskip

Now we consider $p=0$. Let $x$ be in $\g$, $v$ in $V$ and $s$ in $\mathcal{S}$. Applying \eqref{widehat} to $\D(v \otimes s)(x \otimes 1)$, we get:
\begin{equation}\label{C1}
\D(v \otimes s)(x \otimes 1)=\D(v \otimes s)_1(x \otimes 1)\otimes s_1+\ldots+\D(v \otimes s)_m(x \otimes 1)\otimes s_m,
\end{equation}
where $\D(v \otimes s)_j=\chi_j(\D(v \otimes s))$ for all $j$. By definition of $\D$ and $R$ (see \eqref{representacion current} and \eqref{Dif}), it follows:
\begin{equation}\label{C2}
\begin{split}
\D(v \otimes s)(x \otimes 1)&=R(x \otimes 1)(v \otimes s)=\rho(x)(v) \otimes s\\
\,&=\omega_1(s)\rho(x)(v) \otimes s_1+\ldots+\omega_m(s)\rho(x)(v) \otimes s_m,
\end{split}
\end{equation} 
being that $s\!=\!\omega_1(s)s_1\!+\!\ldots\!+\!\omega_m(s)s_m$. Comparing \eqref{C1} and \eqref{C2} we get:
\begin{equation}\label{C3}
\omega_j(s)\rho(x)(v)=\D(v \otimes s)_j(x \otimes 1)=\chi_j(\D(v \otimes s))(x \otimes 1),\,\text{ for all }j.
\end{equation}
By definition of $\mathcal{L}$ (see \eqref{j2}) we obtain:
\begin{equation}\label{C4}
\chi_j(\D(v \otimes s))(x \otimes 1)=\mathcal{L}\left(\chi_j(\D(v \otimes s))\right)(x).
\end{equation}
As $\chi_j(v \otimes s)=\omega_j(s)v$ and $\mathcal{L}(v)=v$, then $\mathcal{L} \circ \chi_j(v \otimes s)=\omega_j(s)v$; hence,
\begin{equation}\label{C5}
\omega_j(s)\rho(x)(v)\!=\!\omega_j(s)\!\d(v)(x)\!=\!\d(\omega_j(s)(v))(x)\!=\!\d\!\left(\mathcal{L}(\chi_j(v \!\otimes \!s)\right)\!(x).
\end{equation}
By \eqref{C3}, \eqref{C4} and \eqref{C5} we get: $\mathcal{L}\left(\chi_j(\D(v\! \otimes \!s))\right)(x)\!=\!\d\!\left(\mathcal{L}(\chi_j(v\! \otimes \!s)\right)\!(x)$, then $\mathcal{L} \circ \chi_j \circ \D=\d \circ \mathcal{L} \circ \chi_j$, for all $j$, and diagram \eqref{L y d conmutan} is commutative.
\end{proof}

\section{The map $\mathcal{T}:\mathcal{C}(\g) \to \mathcal{C}(\g \otimes \mathcal{S})$}

Let $\mathcal{C}(\g)$ be the set of cochains of $\g$, that is: $\mathcal{C}(\g)=\{C(\g;V)\,\vert\,V \text{ is a }\g-\text{module}\}$. We define a map $\mathcal{T}:\mathcal{C}(\g) \to \mathcal{C}(\g \otimes \mathcal{S})$ by:
\begin{equation}\label{j4}
\begin{array}{rl}
\mathcal{T}(C^p(\g;V))&=C^p(\g \otimes \mathcal{S};V \otimes \mathcal{S}),\quad \,\text{ for }p >0,\,\text{ and }\\
\mathcal{T}(V)&=V \otimes \mathcal{S},\quad\,\text{ where }\,V \text{ is a }\g-\text{module}.
\end{array}
\end{equation}
From now on, we assume that $x,x_1,\ldots,x_{p+1}$ are in $\g$; $s,t,t_1,\ldots,t_{p+1}$ are in $\mathcal{S}$; $u$ is in $U$, $v$ is in $V$, $w$ is in $W$; $U,V,W$ are finite dimensional $\g$-modules and $p\!>\!0$. We also consider any cochain $\Lambda$ as in \eqref{widehat}.
\smallskip

For $t_1,\ldots,t_{p}$ in $\mathcal{S}$, we write $\bar{t}$ to denote the product $t_1 \cdots t_p$ in $\mathcal{S}$. 
\smallskip

Let $f:C^p(\g;V) \to C^p(\g;W)$ be a linear map. We define the linear map $\mathcal{T}(f):C^p(\g \otimes \mathcal{S};V \otimes \mathcal{S}) \to C^p(\g \otimes \mathcal{S};W \otimes \mathcal{S})$ by:
\begin{equation}\label{j5}
\mathcal{T}(f)(\Lambda)(x_1 \otimes t_1,\ldots,x_p \otimes t_p)=\sum_{j=1}^mf\left(\mathcal{L}(\Lambda_j)\right)(x_1,\ldots,x_p) \otimes s_j\, \bar{t},
\end{equation}
Let $f:V \to C^p(\g;W)$ be a linear map. We define the linear map $\mathcal{T}(f):V \otimes \mathcal{S} \to C^p(\g \otimes \mathcal{S};V \otimes \mathcal{S})$ by:
\begin{equation}\label{j6}
\mathcal{T}(f)(v \otimes s)(x_1 \otimes t_1,\ldots,x_p \otimes t_p)=f(v)(x_1,\ldots,x_p) \otimes s \,\bar{t},
\end{equation}
Let $f:C^p(\g;V) \to W$ be a linear map. We define the linear map $\mathcal{T}(f):C^p(\g \otimes \mathcal{S};V \otimes \mathcal{S}) \to W \otimes \mathcal{S}$ by:
\begin{equation}\label{j7}
\mathcal{T}(f)(\Lambda)=f(\mathcal{L}(\Lambda_1)) \otimes s_1+\ldots+f(\mathcal{L}(\Lambda_m)) \otimes s_m.
\end{equation}

Let $f\!:\!V \!\to \!W$ be a linear map, then $\mathcal{T}(f)\!:\!V \!\otimes \!\mathcal{S}\! \to \!W \!\otimes \!\mathcal{S}$ is the map: 
\begin{equation}\label{trivial}
\mathcal{T}(f)(v \otimes s)=(f \otimes \mathcal{S})(v \otimes s)=f(v) \otimes s
\end{equation}

\begin{Remark}\label{remark}{\rm 
If $\Id$ is the identity on $C^p(\g;V)$ then $\mathcal{T}(\Id)$ is not necessarily the identity on $C^p(\g \otimes \mathcal{S};V \otimes \mathcal{S})$. Indeed, let $\Lambda$ be in $C^p(\g \otimes \mathcal{S};V \otimes \mathcal{S})$. By definition of $\mathcal{L}$ and \eqref{j5} it follows:
\begin{equation}\label{notFunctor}
\begin{split}
& \mathcal{T}(\Id)(\Lambda)(x_1 \!\otimes \!t_1,\ldots,x_p \!\otimes \!t_p)\!=\!\sum_{j=1}^m \Id \left(\mathcal{L}(\Lambda_j)\right)(x_1,\ldots,x_p)\! \otimes \!s_j\, \bar{t}\\
\,&=\sum_{j=1}^m \mathcal{L}(\Lambda_j)(x_1,\ldots,x_p)\! \otimes\! s_j \bar{t}\!=\!\sum_{j=1}^m \Lambda_j(x_1 \!\otimes \!1,\ldots,x_p \!\otimes \!1) \!\otimes \!s_j\, \bar{t}.
\end{split}
\end{equation}
Then $\mathcal{T}(\Id)(\Lambda)=\Lambda$ if and only if 
\begin{equation}\label{ss1}
\Lambda(x_1 \otimes t_1,\ldots,x_p \otimes t_p)=\sum_{j=1}^m\Lambda_j(x_1 \otimes 1,\ldots,x_p \otimes 1) \otimes s_j\,\bar{t}.
\end{equation}
We will use this remark in the proof of \textbf{Prop. \ref{semisimple}}.}
\end{Remark}

In the next result we prove that $\mathcal{T}$ preserves the composition of maps. 

\begin{Prop}\label{composicion}{\sl
If $f:X \to Y$ and $g:Y \to Z$ are maps in $\mathcal{C}(\g)$, then $\mathcal{T}(g \circ f)=\mathcal{T}(g) \circ \mathcal{T}(f):\mathcal{T}(X) \to \mathcal{T}(Z)$ is a map in $\mathcal{C}(\g \otimes \mathcal{S})$.}
\end{Prop}
We shall verify that if $f:X \to Y$ and $g:Y \to Z$ are maps in $\mathcal{C}(\g)$, then $\mathcal{T}(g \circ f)$ is equal to $\mathcal{T}(g) \circ \mathcal{T}(f)$. Several cases should be considered and we only will prove two of them. In the \textbf{Appendix} we prove the remaining cases.

\begin{claim}{\sl
Let $f:C^p(\g;U) \to C^p(\g;V)$ and $g:C^p(\g;V) \to C^p(\g;W)$ be maps, then $\mathcal{T}(g \circ f)=\mathcal{T}(g) \circ \mathcal{T}(f)$ is a map between $C^p(\g \otimes \mathcal{S};U \otimes \mathcal{S})$ and $C^p(\g \otimes \mathcal{S};W \otimes \mathcal{S})$.
}
\end{claim}
\begin{proof}
Let $\Lambda$ be in $C^p(\g\! \otimes \!\mathcal{S};U\! \otimes \!\mathcal{S})$, and $\Theta\!=\!\mathcal{T}(f)(\Lambda)$. By \eqref{j5} we have:
\begin{equation}\label{expresion para mu}
\mathcal{T}(g)(\Theta)(x_1 \otimes t_1,\cdots,x_p \otimes t_p)=\sum_{j=1}^mg( \mathcal{L}(\Theta_j))(x_1,\ldots,x_p) \otimes s_j\, \bar{t},
\end{equation}
where $\Theta_j=\widehat{\omega}_j \circ \mathcal{T}(f)(\Lambda)$ (see \eqref{widehat}). We claim that $\mathcal{L}(\Theta_j)=f(\mathcal{L}(\Lambda_j))$. Indeed, using the definition of $\mathcal{L}$ and applying \eqref{j5} to $\mathcal{T}(f)(\Lambda)$, we get:
\begin{equation*}\label{expresion para L}
\begin{split}
&\mathcal{L}(\Theta_j)(x_1,\ldots,x_p)=\Theta_j(x_1 \otimes 1,\ldots,x_p \otimes 1)\\
&=\widehat{\omega}_j\circ \mathcal{T}(f)(\Lambda)(x_1 \otimes 1,\ldots,x_p \otimes 1)\\
&=\widehat{\omega}_j\left(\sum_{k=1}^m f \left(\mathcal{L}(\Lambda_k)\right)(x_1,\ldots,x_p) \otimes s_k\right)=f(\mathcal{L}(\Lambda_j))(x_1,\ldots,x_p).
\end{split}
\end{equation*}
Then $\mathcal{L}(\Theta_j)=f(\mathcal{L}(\Lambda_j))$, for all $j$. Substituting this in \eqref{expresion para mu}, we obtain: 
$$
\aligned
\mathcal{T}(g)\!\circ \!\mathcal{T}(f)(\Lambda)(x_1\! \otimes \!t_1,\!\ldots\!,x_p \!\otimes \!t_p)&=\!\!\sum_{j=1}^m (g \circ f)( \mathcal{L}(\Lambda_j))(x_1,\ldots,x_p) \!\otimes \!s_j\, \bar{t},\\
\,&=\mathcal{T}(g \circ f)(\Lambda)(x_1\! \otimes \!t_1,\cdots,x_p \!\otimes \!t_p).
\endaligned
$$
In the last step above we use \eqref{j5}. Thus $\mathcal{T}(g \circ f)\!=\!\mathcal{T}(g) \circ \mathcal{T}(f)$. 
\end{proof}

\begin{claim}{\sl
Let $f:U \to C^p(\g;V)$ and let $g:C^p(\g;V) \to W$ be maps. Then $\mathcal{T}(g \circ f)=\mathcal{T}(g) \circ \mathcal{T}(f)$ is a map between $U \otimes \mathcal{S}$ and $W \otimes \mathcal{S}$.
}
\end{claim}
\begin{proof}
Since $g \!\circ \!f$ is a map between $U$ and $W$, by \eqref{trivial}, $\mathcal{T}(g \!\circ \!f)$ is the map between $U\! \otimes \!\mathcal{S}$ and $W\!\otimes \!\mathcal{S}$ given by $\mathcal{T}(g \!\circ \!f)(u\! \otimes \!s)\!=\!g(f(u))\!\otimes \!s$.
\smallskip

Let $\Lambda=\mathcal{T}(f)(u \otimes s)$ which belongs to $C^p(\g \otimes \mathcal{S};V \otimes  \mathcal{S})$. By \eqref{j6}, $\Lambda(x_1 \otimes 1,\ldots,x_p \otimes 1)\!=\!f(u)(x_1,\ldots,x_p) \otimes s$. By definition of $\mathcal{L}$ we have:
$$
\aligned
& \mathcal{L}\left(\widehat{\omega}_j \circ \Lambda \right)(x_1,\ldots,x_p)=(\widehat{\omega}_j\circ \Lambda)(x_1 \otimes 1,\ldots,x_p \otimes 1)\\
\,&=\widehat{\omega}_j(\Lambda(x_1 \otimes 1,\ldots,x_p \otimes 1))=\omega_j(s)f(u)(x_1,\ldots,x_p),
\endaligned
$$
hence $\mathcal{L}(\widehat{\omega}_j\circ \Lambda)=\omega_j(s)f(u)$ for all $j$. Since $\Lambda_j=\widehat{\omega}_j\circ \Lambda$ and $s=\omega_1(s)s_1+\ldots+\omega_m(s)s_m$, from \eqref{j7} it follows:
\begin{equation}\label{W1}
\begin{split}
\mathcal{T}(g)(\Lambda)&=g(\mathcal{L}(\widehat{\omega}_1\circ\Lambda))\otimes s_1+\ldots+g(\mathcal{L}(\widehat{\omega}_j\circ\Lambda))\otimes s_m\\
\,&=\omega_1(s)g(f(u))\otimes s_1+\ldots+\omega_m(s)g(f(u)) \otimes s_m\\
\,&=g(f(u))\otimes s=\mathcal{T}(g \circ f)(u \otimes s),
\end{split}
\end{equation}
As $\Lambda=\mathcal{T}(f)(u \otimes s)$, from \eqref{W1} we deduce $\mathcal{T}(g)\circ \mathcal{T}(f)=\mathcal{T}(g \circ f)$.
\end{proof}

\begin{Prop}\label{map of complexes}{\sl
Let $f:C(\g;V) \to C(\g;W)$ be a map of complexes, that is $\d \circ f=f \circ \d$ and $f(C^k(\g;V))\subset C^k(\g;W)$ for all $k\geq 0$. Then $\mathcal{T}(f):C(\g \otimes \mathcal{S};V \otimes \mathcal{S}) \to C(\g \otimes \mathcal{S};W \otimes \mathcal{S})$ is a map of complexes.}
\end{Prop}
\begin{proof} To shorten the length of the mathematical expressions, we will use the following notation:
\begin{equation}\label{notacion}
\begin{split}
& \textbf{x} \otimes \textbf{t}=(x_1 \otimes t_1,\ldots,x_{p+1} \otimes t_{p+1}),\\
& (\textbf{x} \otimes \textbf{t})_i=(x_1 \otimes t_1,\ldots,x_{\hat{i}} \otimes t_{\hat{i}},\ldots,x_{p+1} \otimes t_{p+1}),\\
& (\textbf{x} \otimes \textbf{t})_{i,j}=(x_1 \otimes t_1,\ldots,x_{\hat{i}} \otimes t_{\hat{i}},\ldots,x_{\hat{j}} \otimes t_{\hat{j}},\ldots,x_{p+1} \otimes t_{p+1}),\\
& \textbf{x}=(x_1,\ldots,x_{p+1}),\qquad \textbf{x}_{i}=(x_1,\ldots,x_{\hat{i}},\ldots,x_{p+1}),\\
& \textbf{x}_{i,j}=(x_1,\ldots,x_{\hat{i}},\ldots,x_{\hat{j}},\ldots,x_{p+1})
\end{split}
\end{equation}
Let $\Lambda$ be in $C^p(\g \otimes \mathcal{S};V \otimes \mathcal{S})$, $p>0$. We shall prove that $\D\circ \mathcal{T}(f)(\Lambda)=\mathcal{T}(f) \circ \D \Lambda$. Indeed, first we apply $\D$ to $\mathcal{T}(f)(\Lambda)$:
\begin{equation}\label{complex1}
\begin{split}
& \D\mathcal{T}(f)(\Lambda)(\textbf{x} \otimes \textbf{t})=\sum_{i=1}^{p+1}(-1)^{i-1}R(x_i \otimes t_i)\left(\mathcal{T}(f)(\Lambda)\right)((\textbf{x} \otimes \textbf{t})_i)\\
\,&+\sum_{i<j}(-1)^{i+j}\mathcal{T}(f)(\Lambda)\left([x_i \otimes t_i,x_j \otimes t_j],(\textbf{x} \otimes \textbf{t}\right)_{i,j}).
\end{split}
\end{equation}
We write $\textbf{A}$ and $\textbf{B}$, to denote the first and second term in \eqref{complex1}, that is: $\D\mathcal{T}(f)(\Lambda)(\textbf{x} \otimes \textbf{t})\!=\!\textbf{A}\!+\!\textbf{B}$. Applying \eqref{j5} to $\mathcal{T}(f)(\Lambda)$ in $\textbf{A}$, we obtain:
\begin{equation}\label{complex2}
\textbf{A}=\sum_{i=1}^{p+1}\sum_{k=1}^m (-1)^{i-1}R(x_i \otimes t_i)\left(f(\mathcal{L}(\Lambda_k))(\textbf{x}_i) \otimes s_k \,t^{i}\right),
\end{equation}
where $t^{i}=t_1 \cdots t_{i-1}t_{i+1}\cdots t_{p+1}$. In \eqref{complex2} above, we apply $R(x_i \otimes t_i)$ to $f(\mathcal{L}(\Lambda_k))(\textbf{x}_i) \otimes s_k \,t^{i}$ (see \eqref{representacion current}), and we get:
\begin{equation}\label{complex3}
\textbf{A}=\sum_{i=1}^{p+1}\sum_{k=1}^m (-1)^{i-1}\rho(x_i)\left(f\left(\mathcal{L}(\Lambda_k)\right)(\textbf{x}_i\right)\otimes s_k\, \bar{t},
\end{equation}
for $\bar{t}\!=\!t_i t^{i}$. Regarding to $\textbf{B}$, we fix $i\!<\!j$; by \eqref{j5} and \eqref{notacion} we have:
$$
\aligned
& \mathcal{T}(f)(\Lambda)\left([x_i \otimes t_i,x_j \otimes t_j],(\textbf{x}\otimes \textbf{t})_{i,j}\right)\\
&=\!\!\mathcal{T}(f)(\Lambda)\left([x_i,x_j]\otimes t_i t_j,x_1 \otimes t_1,\ldots,x_{\hat{i}} \otimes t_{\hat{i}},\ldots,x_{\hat{j}} \otimes t_{\hat{j}},\ldots,x_{p+1} \otimes t_{p+1}\right)\\
&=\!\!f(\mathcal{T}(\Lambda_1))\left([x_i,x_j],\textbf{x}_{i,j}\right)\otimes s_1\,\overline{t}+\ldots+f(\mathcal{T}(\Lambda_m))\left([x_i,x_j],\textbf{x}_{i,j}\right)\otimes s_m\,\overline{t}.
\endaligned
$$
for $\overline{t}=(t_i t_j)(t_1\cdots t_{\hat{i}}\cdots t_{\hat{j}}\cdots t_{p+1})$. Hence,
\begin{equation}\label{complex4}
\textbf{B}=\sum_{i<j}\sum_{k=1}^m(-1)^{i+j} f\left(\mathcal{L}(\Lambda_k)\right)([x_i,x_j],\textbf{x}_{i,j}) \otimes s_k\, \bar{t}.
\end{equation}
From \eqref{complex3} and \eqref{complex4} it follows:
\begin{equation}\label{complex5}
\D\mathcal{T}(f)(\Lambda)(\textbf{x} \otimes \textbf{t})=\textbf{A}+\textbf{B}=\sum_{k=1}^m \d\left(f\left(\mathcal{L}(\Lambda_k)\right)\right)(\textbf{x}) \otimes s_k \bar{t}.
\end{equation}
By hypothesis $f$ is a map of complex then $f \circ \d=\d \circ f$. By \eqref{j3}, $\d(\mathcal{L}(\Lambda_k))=\mathcal{L}((\D \Lambda)_k)$, then by \eqref{j5} and \eqref{complex5} we get:
\begin{equation*}\label{complex6}
\begin{split}
&\D\mathcal{T}(f)(\Lambda)(\textbf{x} \otimes \textbf{t})=\sum_{k=1}^m f(\d(\mathcal{L}(\Lambda_k)))(\textbf{x}) \otimes s_k\, \bar{t}\\
&=\sum_{k=1}^m f\left(\left(\mathcal{L}(\D\Lambda)_k\right)\right)(\textbf{x}) \otimes s_k\, \bar{t}=\mathcal{T}(f)(\D\Lambda)(\textbf{x} \otimes \textbf{t}).
\end{split}
\end{equation*}
As $\Lambda$ is arbitrary, it follows $\D\circ \mathcal{T}(f)\!=\!\mathcal{T}(f) \circ \D$. 
\smallskip

We shall prove that $\D$ and $\mathcal{T}(f)$ commute for $p=0$. In this case, $f:V \to W$ and $T(f)=f \otimes \mathcal{S}$ (see \eqref{trivial}). We claim that: 
$$
\mathcal{T}(f) \circ \D(v \otimes s)(x \otimes t)=\D \circ \mathcal{T}(f)(v \otimes s)(x \otimes t).
$$ 
Indeed, we first apply $\mathcal{T}(f)$ to $\D(v \otimes s)$ (see \eqref{j5}):
\begin{equation}\label{K1}
\mathcal{T}(f)(\D(v\! \otimes \!s))(x \!\otimes \!t)\!=\!\!\sum_{j=1}^m f\left(\mathcal{L}(\D (v \!\otimes \!s)_j) \right)(x)\! \otimes \!s_j\,t.
\end{equation}
Due to $\D(v \otimes s)_j=\chi_j\circ \D(v \otimes s)$ and  $\mathcal{L} \circ \chi_j \circ \D=\d \circ \mathcal{L} \circ \chi_j$ (see \textbf{Prop. \ref{prop L y d conmutan}}), it follows: 
\begin{equation}\label{h1}
\begin{split}
&\sum_{j=1}^m f\left(\mathcal{L}(\D (v \!\otimes \!s)_j) \right)(x)\! \otimes \!s_j\,t=\!\sum_{j=1}^m f\left(\mathcal{L}\!\circ \!\chi_j\!\circ \!\D (v \!\otimes \!s) \right)(x) \!\otimes \!s_jt\\
&=\!\sum_{j=1}^m f\left(\d \circ \mathcal{L} \circ \chi_j (v \!\otimes \!s) \right)(x)\!\otimes \!s_jt
\end{split}
\end{equation}
As $\chi_j(v \otimes s)=\omega_j(s)v$ and $\mathcal{L}(v)=v$, then $\mathcal{L}\circ \chi_j(v \!\otimes \!s)\!=\!\omega_j(s)v$. Hence, $f(\d \circ \mathcal{L}\circ \chi_j(v\! \otimes\! s))\!=\!\omega_j(s)f(\d(v))$. Substituting this in \eqref{h1} and using that $f$ and $\d$ commute, as well as $s=\omega_1(s)s_1+\ldots+\omega_m(s)s_m$, we get:
\begin{equation}\label{h2}
\begin{split}
& \mathcal{T}(f)\circ \D(v \otimes s)(x \otimes t)\\
&=\omega_1(s)\d f(v)(x)\otimes s_1t+\ldots+\omega_m(s)\d f(v)(x)\otimes s_mt.\\
&=(\d f(v))(x) \otimes s\,t.
\end{split}
\end{equation}
Let $\rho_W$ be the representation of $\g$ on $W$. Since $\d f(v)(x)=\rho_W(x)(f(v))$ and $\D(f(v) \otimes s)(x \otimes t)=\rho_W(x)(f(v)) \otimes s\,t$, from \eqref{h2} it follows:
\begin{equation}\label{h3}
\begin{split}
&\mathcal{T}(f)\circ \D(v \otimes s)(x \otimes t)=\rho_W(x)(f(v))\otimes s\,t\\
&=\D (f(v)\otimes s)(x \otimes t)=\D \left(\mathcal{T}(f)(v \otimes s)\right)(x \otimes t),
\end{split}
\end{equation}
because $\mathcal{T}(f)(v \otimes s)=f(v) \otimes s$ (see \eqref{trivial}). Hence, $\mathcal{T}(f)\circ \D=\D \circ \mathcal{T}(f)$.
\end{proof}

\section{The map $\alpha:\mathcal{H}(\g \otimes \mathcal{S};V \otimes \mathcal{S})\to \mathcal{H}(\g ;V) \otimes \mathcal{S}$}

Let $V$ be a $\g$-module. We denote the group of cocycles and coboundaries of $C(\g;V)$, by $\mathcal{Z}$ and $\mathcal{B}$, respectively. The cohomology group of $\g$ with coefficients in $V$ is denoted by $\mathcal{H}(\g;V)$. The quotient $C(\g;V)/\mathcal{B}$ is denoted by $\mathcal{Z}^{\prime}$ and $C(\g;V)/\mathcal{Z}$ is denoted by $\mathcal{B}^{\prime}$. We assume that $\mathcal{Z},\mathcal{B},\mathcal{Z}^{\prime},\mathcal{B}^{\prime}, \mathcal{H}(\g;V)$ and $V$, have zero differential. By \cite{Che} (Chapter IV, \S 23), $\mathcal{Z}$ and $\mathcal{B}$ are $\g$-modules, then $\mathcal{Z}^{\prime},\mathcal{B}^{\prime}$ and $\mathcal{H}(\g;V)$ are $\g$-modules with zero differential. 
\smallskip

Let $\iota:\mathcal{Z}\to C(\g;V)$, be the inclusion. By \eqref{j6} we get a map $\mathcal{T}(\iota):\mathcal{Z} \otimes \mathcal{S}\to C(\g \otimes \mathcal{S};V \otimes \mathcal{S})$. Due to $\mathcal{Z}\otimes \mathcal{S}$ has zero differential, there is a map $\Phi$ from $\mathcal{Z} \otimes \mathcal{S}$ into $\mathcal{H}(\g \otimes \mathcal{S};V \otimes \mathcal{S})$, induced by $\mathcal{T}(\iota)$:
\begin{equation}\label{phi}
\begin{array}{rccl}
\Phi:& \mathcal{Z}\otimes \mathcal{S}& \longrightarrow & \mathcal{H}(\g \otimes \mathcal{S};V \otimes \mathcal{S})\\
& \bar{x} & \mapsto & \mathcal{T}(\iota)(\bar{x})+\mathcal{B}(\g \otimes \mathcal{S};V \otimes \mathcal{S}).
\end{array}
\end{equation} 
Consider $\pi^{\prime}:C(\g;V)\to \mathcal{Z}^{\prime}$, defined by $\pi^{\prime}(\lambda)=\lambda+\mathcal{B}$. By \eqref{j7}, we get a map $\mathcal{T}(\pi^{\prime}):C(\g \otimes \mathcal{S};V \otimes \mathcal{S})\to \mathcal{Z}^{\prime} \otimes \mathcal{S}$. Due to $\mathcal{Z}^{\prime} \otimes \mathcal{S}$ has zero differential, we have a map $\Psi$ from $\mathcal{H}(\g \otimes \mathcal{S};V \otimes \mathcal{S})$ into $\mathcal{Z}^{\prime} \otimes \mathcal{S}$, induced by $\mathcal{T}(\pi^{\prime})$:
\begin{equation}\label{psi}
\begin{array}{rccl}
\Psi:&\!\!\!\!\!\!\!\!\!\!\!\!\!\mathcal{H}(\g \!\otimes \!\mathcal{S};V \!\otimes \!\mathcal{S})\!\!\!\!\!& \!\!\!\!\!\!\!\!\!\!\longrightarrow & \mathcal{Z}^{\prime}\! \otimes \!\mathcal{S}\\
& \Lambda+\mathcal{B}(\g \!\otimes \!\mathcal{S};V \!\otimes \!\mathcal{S}) & \mapsto\!\!\!\! & \mathcal{T}(\pi^{\prime})(\Lambda)\!\!=\!\!\displaystyle{\sum_{j=1}^m}\left(\mathcal{L}(\Lambda_j)+\mathcal{B}\right)\!\otimes \!s_j.
\end{array}
\end{equation} 
We shall prove that $\Psi$ is well-defined. Let $\sigma$ be in $C(\g \otimes \mathcal{S};V \otimes \mathcal{S})$. Using \eqref{j3} and \eqref{j7}, as well as $\pi^{\prime}\circ \d=0$, we obtain:
$$
\aligned
\mathcal{T}(\pi^{\prime})(\D\sigma)&=\pi^{\prime}\left(\mathcal{L}((\D\sigma)_1)\right) \otimes s_1+\ldots+\pi^{\prime}\left(\mathcal{L}((\D\sigma)_m)\right) \otimes s_m\\
\,&=\pi^{\prime}\left(\d\mathcal{L}(\sigma_1)\right) \otimes s_1+\ldots+\pi^{\prime}\left(\d\mathcal{L}(\sigma_m)\right) \otimes s_m=0,
\endaligned
$$
Then $\mathcal{B}(\g \otimes \mathcal{S};V \otimes \mathcal{S}) \subset \Ker(\mathcal{T}(\pi^{\prime}))$, hence $\Psi$ is well-defined.
\smallskip

Let $\pi:\mathcal{Z}\to \mathcal{H}(\g;V)$ be the projection and $\iota^{\prime}:\mathcal{H}(\g;V)\to \mathcal{Z}^{\prime}$ be the inclusion. In the next result we will prove that there exists a surjective linear map $\alpha$ between $\mathcal{H}(\g \otimes \mathcal{S};V \otimes \mathcal{S})$ and $\mathcal{H}(\g;V) \otimes \mathcal{S}$.

\begin{Prop}\label{morfismo natural}{\sl
Let $\g \otimes \mathcal{S}$ be the current Lie algebra of $\g$.
\smallskip

\textbf{(i)} For any $\g$-module $V$, there exists a unique surjective linear map $\alpha\!:\!\!\mathcal{H}(\g \otimes \mathcal{S};V \otimes \mathcal{S})\!\to \!\mathcal{H}(\g;V) \otimes \mathcal{S}$, that makes commutative the diagram:
\begin{equation}\label{propiedad de alpha}
\xymatrix@R+1pc@C+2pc{
\mathcal{Z} \otimes \mathcal{S} \ar[r]^{\pi \otimes \mathcal{S}} \ar[d]_{\Phi} & \mathcal{H}(\g;V) \otimes \mathcal{S} \ar[d]^{\iota^{\prime} \otimes \mathcal{S}}\\
\mathcal{H}(\g \otimes \mathcal{S};V \otimes \mathcal{S}) \ar[r]^{\Psi} \ar[ur]^{\alpha} & \mathcal{Z}^{\prime} \otimes \mathcal{S}
}
\end{equation}

\textbf{(ii)} Let $f:C(\g;V) \to C(\g;W)$ be a map of complexes and consider $\mathcal{H}(f)\!:\mathcal{H}(\g;V) \!\to \!\mathcal{H}(\g;W)$ the map induced by $f$; that is $\mathcal{H}(f)(\lambda+\mathcal{B})=f(\lambda)+\mathcal{B}$. Then the following diagram is commutative:
\begin{equation}\label{u3}
\xymatrix@C+2.5pc{
\mathcal{H}(\g \otimes \mathcal{S};V \otimes \mathcal{S}) \ar[r]^{\mathcal{H}\left(\mathcal{T}(f)\right)} \ar[d]_{\alpha_V} & \mathcal{H}(\g \otimes \mathcal{S};W \otimes \mathcal{S}) \ar[d]^{\alpha_W}\\
\mathcal{H}(\g;V) \otimes \mathcal{S} \ar[r]^{\mathcal{H}(f) \otimes \mathcal{S}=\mathcal{T}(\mathcal{H}(f))} & \mathcal{H}(\g;W) \otimes \mathcal{S}
}
\end{equation}
where $\mathcal{H}(\mathcal{T}(f))$ is the map induced by $\mathcal{T}(f)$.}
\end{Prop}

\begin{proof}

\textbf{(i)} First we shall prove the following.

\begin{claim}\label{T1}{\sl
For $p=0$, $\mathcal{H}^0(\g \otimes \mathcal{S};V \otimes \mathcal{S})=\mathcal{H}^0(\g;V) \otimes \mathcal{S}$.}
\end{claim}
\begin{proof}
Observe that for $p=0$, $\mathcal{Z}^0=\mathcal{H}^0(\g;V)$. Consider $\bar{v}$ in $\mathcal{H}^0(\g \otimes \mathcal{S};V \otimes \mathcal{S}) \subset V \otimes \mathcal{S}$. By \textbf{Prop. \ref{prop1}} we write $\bar{v}=v_1 \otimes s_1+\ldots+v_m \otimes s_m$, where $v_j$ belongs to $V$, thus:
\begin{equation}\label{F1}
\begin{split}
0=\D(\bar{v})(x \otimes 1)&=R(x \otimes 1)(\bar{v})\\
&=\rho(x)(v_1) \otimes s_1+\ldots+\rho(x)(v_m)\otimes s_m\\
&=\d(v_1)(x) \otimes s_1+\ldots+\d(v_m)(x) \otimes s_m.
\end{split}
\end{equation}
Then $\d(v_j)=0$, and $v_j$ belongs to $\mathcal{H}^0(\g;V)$ for all $j$, which implies that $\bar{v}$ belongs to $\mathcal{H}^0(\g;V) \otimes \mathcal{S}$. Hence $\mathcal{H}^0(\g \otimes \mathcal{S};V \otimes \mathcal{S}) \subset \mathcal{H}^0(\g;V) \otimes \mathcal{S}$.
\smallskip

Let $\bar{v}$ be in $\mathcal{H}^0(\g;V) \otimes \mathcal{S} \subset V \otimes \mathcal{S}$. By \textbf{Prop. \ref{prop1}}, there are $v_1,\ldots,v_m$ in $\mathcal{H}^0(\g;V)$ such that $\bar{v}=v_1 \otimes s_1+\ldots+v_m \otimes s_m$. As each $v_j$ belongs to $\mathcal{H}^0(\g;v)$, then $\d(v_j)=0$. Thus,
$$
\aligned
& \D(\bar{v})(x \otimes s)=R(x \otimes s)(\bar{v})=\sum_{j=1}^m R(x \otimes s)(v_j \otimes s_j)\\
&=\sum_{j=1}^m \rho(x)(v_j) \otimes s s_j=\sum_{j=1}^m \d(v_j)(x) \otimes s s_j=0.
\endaligned
$$
Hence $\D(\bar{v})=0$, and $\bar{v}$ is in $\mathcal{H}^0(\g \otimes \mathcal{S};V \otimes \mathcal{S})$. Thus $\mathcal{H}^0(\g \otimes \mathcal{S};V \otimes \mathcal{S})=\mathcal{H}^0(\g;V) \otimes \mathcal{S}$.  
\end{proof}

Let $\eta:\mathcal{B}\to \mathcal{Z}$ be the inclusion map and $\zeta:\mathcal{Z}^{\prime}\to \mathcal{B}^{\prime}$ be the map defined by $\zeta(\lambda+\mathcal{B})=\lambda+\mathcal{Z}$. We have the following short exact sequences:
$$
\aligned
& \xymatrix{
0 \ar[r] & \mathcal{H}(\g;V) \ar[r]^<(.3){\iota^{\prime}} & \mathcal{Z}^{\prime} \ar[r]^{\zeta} & \mathcal{B}^{\prime} \ar[r] & 0
}\\
& \xymatrix{
0 \ar[r] & \mathcal{B} \ar[r]^{\eta} & \mathcal{Z} \ar[r]^<(.25){\pi} & \mathcal{H}(\g;V) \ar[r] & 0
}
\endaligned
$$
Since $\mathcal{S}$ is finite dimensional, the following sequences are exact:
$$
\aligned
& \xymatrix{
0 \ar[r] & \mathcal{H}(\g;V)\otimes \mathcal{S} \ar[r]^<(.3){\iota^{\prime}\otimes \mathcal{S}} & \mathcal{Z}^{\prime} \otimes \mathcal{S} \ar[r]^{\zeta\otimes \mathcal{S}} & \mathcal{B}^{\prime}\otimes \mathcal{S} \ar[r] & 0
}\\
& \xymatrix{
0 \ar[r] & \mathcal{B} \otimes \mathcal{S} \ar[r]^{\eta \otimes \mathcal{S}} & \mathcal{Z}\otimes \mathcal{S} \ar[r]^<(.2){\pi\otimes \mathcal{S}} & \mathcal{H}(\g;V) \otimes \mathcal{S}\ar[r] & 0
}
\endaligned
$$
Thus $\pi \otimes \mathcal{S}$ is surjective and $\iota^{\prime} \otimes \mathcal{S}$ is injective. Due to $\iota^{\prime} \circ \pi=\pi^{\prime} \circ \iota$, by \textbf{Prop. \ref{composicion}} we obtain the commutative diagram:
\begin{equation}\label{1x1}
\xymatrix@C+2pc{
\mathcal{Z} \otimes \mathcal{S} \ar[r]^{\pi \otimes \mathcal{S}} \ar[d]_{\mathcal{T}(\iota)} & \mathcal{H}(\g;V) \otimes \mathcal{S} \ar[d]^{\iota^{\prime} \otimes \mathcal{S}}\\
C(\g \otimes \mathcal{S};V \otimes \mathcal{S}) \ar[r]^>(.75){\mathcal{T}(\pi^{\prime})} & \mathcal{Z}^{\prime} \otimes \mathcal{S}
}
\end{equation}
By \eqref{phi} and \eqref{psi} we have the commutative diagram:
\begin{equation}\label{2x}
\xymatrix@C+2pc{
\,& 0 \ar[d]\\
\mathcal{Z} \otimes \mathcal{S} \ar[r]^{\pi \otimes \mathcal{S}} \ar[d]_{\Phi} & \mathcal{H}(\g;V) \otimes \mathcal{S} \ar[d]^{\iota^{\prime} \otimes \mathcal{S}} \ar[r] & 0\\
\mathcal{H}(\g \otimes \mathcal{S};V \otimes \mathcal{S}) \ar[r]^{\Psi}  & \mathcal{Z}^{\prime} \otimes \mathcal{S}
}
\end{equation}

If $\alpha$ and $\alpha^{\prime}$ make commutative the diagram \eqref{2x}, then $\left(\iota^{\prime} \otimes \mathcal{S}\right) \circ \alpha=\left(\iota^{\prime} \otimes \mathcal{S}\right) \circ \alpha^{\prime}$. As $\iota^{\prime} \otimes \mathcal{S}$ is injective, it follows $\alpha=\alpha^{\prime}$. 
\smallskip

If $\operatorname{Im}(\Psi) \subset \operatorname{Im}\left(\iota^{\prime} \otimes \mathcal{S}\right)=\operatorname{Ker}\left(\zeta \otimes \mathcal{S}\right)$, then there exists a map $\alpha$ between $\mathcal{H}(\g \otimes \mathcal{S};V \otimes \mathcal{S})$ and $\mathcal{H}(\g;V) \otimes \mathcal{S}$. We will prove this affirmation.

\begin{claim}\label{nueva claim}{\sl 
The composition $(\zeta \otimes \mathcal{S}) \circ \Psi$ is zero.}
\end{claim}
\begin{proof}
For $p=0$, $T(\iota)=\iota \otimes \mathcal{S}$, as $\iota:\mathcal{Z} \to C(\g;V)$ and $V=C^0(\g;V)$ (see \eqref{trivial}). Then,
$$
\Im(\mathcal{T}(\iota))=\Im(\iota \otimes \mathcal{S})=\mathcal{Z}^0 \otimes \mathcal{S}=\mathcal{H}^0(\g;V)\otimes \mathcal{S}=\mathcal{H}^0(\g \otimes \mathcal{S};V \otimes \mathcal{S}).
$$
Therefore by \eqref{phi}, $\Phi^0:\mathcal{Z}^0 \otimes \mathcal{S}\to \mathcal{H}^0(\g \otimes \mathcal{S};V \otimes \mathcal{S})$ is the identity.
\smallskip

Similarly, for $p=0$, $\pi^{\prime}:C(\g;V) \to \mathcal{Z}^{\prime}$ is the identity, for $\mathcal{B}^0=\{0\}$ and ${\mathcal{Z}^{\prime}}^0\!\!=\!V$. By \eqref{trivial}, $\mathcal{T}(\pi^{\prime})=\Id_V \otimes \mathcal{S}=\Id_{V \otimes \mathcal{S}}$ is the identity on $V \otimes \mathcal{S}$. Then $\Psi^0:\!\mathcal{H}^0(\g \otimes \mathcal{S};V \otimes \mathcal{S})\!\to \!{\mathcal{Z}^{\prime}}^0 \otimes \mathcal{S}$ is the inclusion (see \eqref{psi}). 
\smallskip

In addition, for $p=0$ the map $\zeta:\mathcal{Z}^{\prime} \to \mathcal{B}^{\prime}$ is the projection. If $\bar{v}=v_1 \otimes s_1+\ldots+v_m \otimes s_m$ belongs to $\mathcal{H}^0(\g \otimes \mathcal{S};V \otimes \mathcal{S})=\mathcal{H}^0(\g;V) \otimes \mathcal{S}$, by \textbf{Prop. \ref{prop1}} we may assume that each $v_j$ is in $\mathcal{H}^0(\g;V)=\mathcal{Z}^0$, then:
$$
\aligned
&(\zeta \otimes \mathcal{S})\circ \Psi)(\bar{v})=(\zeta \otimes \mathcal{S})(v_1 \otimes s_1+\ldots+v_m \otimes s_m)\\
&=\zeta(v_1) \otimes s_1+\ldots+\zeta(v_m) \otimes s_m=0,
\endaligned
$$
Thus $\operatorname{Im}(\Psi) \subset \operatorname{Ker}\left(\zeta\otimes \mathcal{S}\right)=\Im(\iota^{\prime} \otimes \mathcal{S})$, and $(\zeta \otimes \mathcal{S})\circ \Psi=0$. 
\smallskip

For $p=0$ the map $\pi \otimes \mathcal{S}:\mathcal{Z}^0\otimes \mathcal{S}\to \mathcal{H}^0(\g;V)\otimes \mathcal{S}$ is the identity and $\iota^{\prime}\otimes \mathcal{S}:\mathcal{H}^0(\g;V)\otimes \mathcal{S}\to {\mathcal{Z}^{\prime}}^0\otimes \mathcal{S}$ is the inclusion. Since any map that makes commutative \eqref{propiedad de alpha} is unique, then $\alpha$ is the identity for $p=0$.
\smallskip

For $p>0$, let $\Lambda$ be in $C^p(\g \otimes \mathcal{S};V \otimes \mathcal{S})$ such that $\D \Lambda=0$. By \eqref{j3}, $\mathcal{L}(\Lambda_j)$ belongs to $\mathcal{Z}$ for all $j$. By \eqref{trivial} and \eqref{psi} we have:
$$
\aligned
& \left(\zeta \otimes \mathcal{S}\right) \circ \Psi\left(\Lambda+\mathcal{B}(\g \otimes \mathcal{S};V \otimes \mathcal{S}) \right)\\
&=(\zeta\otimes \mathcal{S})\left((\mathcal{L}(\Lambda_1)+\mathcal{B}) \otimes s_1+\ldots+(\mathcal{L}(\Lambda_m)+\mathcal{B}) \otimes s_m \right)\\
&=\left(\mathcal{L}(\Lambda_1)+\mathcal{Z}\right) \otimes s_1+\ldots+\left(\mathcal{L}(\Lambda_m)+\mathcal{Z}\right) \otimes s_m=0.
\endaligned
$$
Then the composition $\left(\zeta \otimes \mathcal{S}\right) \circ \Psi$ is zero for $p>0$.

\end{proof}

\begin{claim}\label{claim 5}{\sl
There exists a linear map $\alpha:\mathcal{H}(\g \otimes \mathcal{S};V \otimes \mathcal{S}) \to \mathcal{H}(\g;V) \otimes \mathcal{S}$ that makes commutative the diagram \eqref{propiedad de alpha}.
}
\end{claim}
\begin{proof}
Since $\Im(\Psi) \subset \Ker(\zeta \otimes \mathcal{S})=\Im(\iota^{\prime} \otimes \mathcal{S})$ (see \textbf{Claim. \ref{nueva claim}}), then for each $\overline{\Lambda}$ in $\mathcal{H}(\g \otimes \mathcal{S};V \otimes \mathcal{S})$ there exists $\theta$ in $\mathcal{H}(\g;V) \otimes \mathcal{S}$, such that $\Psi(\overline{\Lambda})=\left(\iota^{\prime} \otimes \mathcal{S}\right)(\theta)$. Due to $\iota^{\prime} \otimes \mathcal{S}$ is injective, $\theta$ is unique. 
\smallskip

Let $\alpha:\mathcal{H}(\g \otimes \mathcal{S};V \otimes \mathcal{S}) \to \mathcal{H}(\g;V) \otimes \mathcal{S}$, be defined by $\alpha(\overline{\Lambda})=\theta$, then $\Psi=(\iota^{\prime}\otimes \mathcal{S})\circ \alpha$. Due to $\iota^{\prime} \otimes \mathcal{S}$ is injective, then $\Ker(\alpha)=\Ker(\Psi)$. Since \eqref{2x} is commutative, it follows that $(\iota^{\prime} \otimes \mathcal{S})\circ \left(\alpha \circ \Phi\right)\!=\!(\iota^{\prime} \otimes \mathcal{S}) \circ (\pi \otimes \mathcal{S})$. Thus, $\alpha\circ\Phi=\pi \otimes \mathcal{S}$, and $\alpha$ makes commutative the diagram \eqref{propiedad de alpha}.
\end{proof}

\begin{claim}{\sl
The map $\alpha:\mathcal{H}(\g \otimes \mathcal{S};V \otimes \mathcal{S}) \to \mathcal{H}(\g;V) \otimes \mathcal{S}$ is surjective.}
\end{claim}
\begin{proof}
Let $\theta$ be in $\mathcal{H}(\g;V)\otimes\mathcal{S}$. Since $\pi \otimes \mathcal{S}$ is surjective, there exists $\mu$ in $\mathcal{Z} \otimes \mathcal{S}$ such that $\left(\pi \otimes \mathcal{S}\right)(\mu)=\theta$. Let $\bar{\Lambda}=\Phi(\mu)$, then $\alpha(\bar{\Lambda})=(\alpha \circ \Phi)(\mu)=(\pi \otimes \mathcal{S})(\mu)=\theta$. Thus $\alpha$ is surjective. 
\end{proof}

We shall give an explicit description of $\alpha$. Let $\Lambda+\mathcal{B}(\g \otimes \mathcal{S};V \otimes \mathcal{S})$ be in $\mathcal{H}(\g \otimes \mathcal{S};V \otimes \mathcal{S})$, where $\Lambda$ is in $C^p(\g \otimes \mathcal{S};V \otimes \mathcal{S})$. As $\alpha(\Lambda+\mathcal{B}(\g \otimes \mathcal{S};V \otimes \mathcal{S}))$ belongs to $\mathcal{H}(\g;V) \otimes \mathcal{S}$, by \textbf{Prop. \ref{prop1}} there are $\mu_j$ in $\mathcal{Z}$ such that:
\begin{equation}\label{explicita 1}
\alpha(\Lambda+\mathcal{B}(\g \otimes \mathcal{S};V \otimes \mathcal{S}))\!\!=\!\!\left(\mu_1+\mathcal{B}\right) \otimes s_1+\dots+\left(\mu_m+\mathcal{B}\right) \otimes s_m.
\end{equation}
By \eqref{propiedad de alpha}, $(\iota^{\prime} \otimes \mathcal{S})\circ \alpha=\Psi$. Then using \eqref{j7} and \eqref{psi} it follows:
\begin{equation}\label{explicita2}
\begin{split}
&(\iota^{\prime} \otimes \mathcal{S})\circ \alpha(\Lambda+\mathcal{B}(\g \otimes \mathcal{S};V \otimes \mathcal{S}))=\Psi\left(\Lambda+\mathcal{B}(\g \otimes \mathcal{S};V \otimes \mathcal{S})\right)\\
&=\left(\mathcal{L}(\Lambda_1)+\mathcal{B}\right)\otimes s_1+\ldots+\left(\mathcal{L}(\Lambda_m)+\mathcal{B}\right)\otimes s_m.
\end{split}
\end{equation}
Applying $\iota^{\prime} \otimes \mathcal{S}$ to \eqref{explicita 1}, we get:
\begin{equation}\label{explicita3}
\begin{split}
&(\iota^{\prime} \otimes \mathcal{S})\circ \alpha(\Lambda+\mathcal{B}(\g \otimes \mathcal{S};V \otimes \mathcal{S}))\\
&=(\iota^{\prime} \otimes \mathcal{S})\left(\sum_{j=1}^m\left( \mu_j+\mathcal{B}\right) \otimes s_j \right)=\sum_{j=1}^m\left( \mu_j+\mathcal{B}\right) \otimes s_j
\end{split}
\end{equation}
From \eqref{explicita2} and \eqref{explicita3} it follows: $\mu_j+\mathcal{B}=\mathcal{L}(\Lambda_j)+\mathcal{B}$ for all $j$. Thus, by \eqref{explicita 1} we obtain:
\begin{equation}\label{explicita alfa}
\alpha(\Lambda+\mathcal{B}(\g \otimes \mathcal{S};V \otimes \mathcal{S}))\!=\!\left( \mathcal{L}(\Lambda_1)+\mathcal{B}\right) \otimes s_1+\ldots+\left( \mathcal{L}(\Lambda_m)+\mathcal{B}\right) \otimes s_m,
\end{equation}
Finally, in the proof of \textbf{Claim \ref{nueva claim}}, we showed that $\alpha$ is the identity on $\mathcal{H}^0(\g \otimes \mathcal{S};V \otimes \mathcal{S})=\mathcal{H}^0(\g;V) \otimes \mathcal{S}$. 
\smallskip

\textbf{(ii)} We shall prove that if $f:C(\g;V) \to C(\g;W)$ is a map of complexes, then the following diagram is commutative:
\begin{equation}\label{propiedad natural de alpha}
\xymatrix@C+1pc{
\mathcal{H}(\g \otimes \mathcal{S};V \otimes \mathcal{S}) \ar[r]^{\mathcal{H}\left(\mathcal{T}(f)\right)} \ar[d]_{\alpha_V} & \mathcal{H}(\g \otimes \mathcal{S};W \otimes \mathcal{S}) \ar[d]^{\alpha_W}\\
\mathcal{H}(\g;V) \otimes \mathcal{S} \ar[r]^{\mathcal{H}(f) \otimes \mathcal{S}} & \mathcal{H}(\g;W) \otimes \mathcal{S}
}
\end{equation}
Let $f^{\prime}:\mathcal{Z}^{\prime}(\g;V)\to \mathcal{Z}^{\prime}(\g;W)$ be the map induced by $f$, that is $f^{\prime}(\lambda+\mathcal{B})=f(\lambda)+\mathcal{B}$. Then $\pi^{\prime} \circ f=f^{\prime} \circ \pi^{\prime}$. Due to the diagram \eqref{propiedad de alpha} is commutative, by \eqref{psi} we have:
\begin{equation}\label{u1}
\begin{split}
&(\iota^{\prime} \otimes \mathcal{S}) \circ (\alpha_W \circ \mathcal{H}(\mathcal{T}(f))\\
&=\left((\iota^{\prime} \otimes \mathcal{S}) \circ \alpha_W\right) \circ \mathcal{H}(\mathcal{T}(f))=\Psi \circ \mathcal{H}(\mathcal{T}(f))=\mathcal{T}(f^{\prime}) \circ \Psi.
\end{split}
\end{equation}
By \eqref{trivial}, $\mathcal{T}(f^{\prime})=f^{\prime} \otimes \mathcal{S}$, and by \eqref{propiedad de alpha}, $\Psi=(\iota^{\prime} \otimes \mathcal{S}) \circ \alpha_V$. In addition, $\mathcal{T}(\iota^{\prime})=\iota^{\prime} \otimes \mathcal{S}$ and as $f$ is a map of complexes, then $\iota^{\prime}\circ \mathcal{H}(f)=f^{\prime}\circ \iota^{\prime}$, hence:
\begin{equation}\label{u2}
\begin{split}
& \mathcal{T}(f^{\prime}) \circ \Psi= \mathcal{T}(f^{\prime})\circ (\mathcal{T}(\iota^{\prime})\circ \alpha_V)=\left(\mathcal{T}(f^{\prime})\circ \mathcal{T}(\iota^{\prime})\right)\circ \alpha_V\\
&=\mathcal{T}\left(f^{\prime} \circ \iota^{\prime}\right) \circ \alpha_V=\mathcal{T}\left(\iota^{\prime} \circ \mathcal{H}(f) \right) \circ \alpha_V\\
&=\left(\left(\iota^{\prime} \circ \mathcal{H}(f) \right) \otimes \mathcal{S} \right) \circ \alpha_V=\left(\iota^{\prime} \otimes \mathcal{S} \right) \circ \left(\left(\mathcal{H}(f) \otimes \mathcal{S} \right) \circ \alpha_V\right)
\end{split}
\end{equation}

As $\iota^{\prime} \otimes \mathcal{S}$ is injective, from \eqref{u1} and \eqref{u2} we deduce that $\alpha_W \circ \mathcal{H}(\mathcal{T}(f))=(\mathcal{H}(f) \otimes \mathcal{S}) \circ \alpha_V$, which proves that \eqref{propiedad natural de alpha} is commutative.
\end{proof}

Let $\mathcal{R}^0=\{0\}$. For each $p>0$, consider the subspace $\mathcal{R}^p$ of $C^p(\g \otimes \mathcal{S};V \otimes \mathcal{S})$, generated by all $\Theta$ such that $\Theta(x_1 \otimes 1,\ldots,x_p \otimes 1)=0$. Let $\mathcal{R}=\oplus_{p\geq 0}\mathcal{R}^p$ and let $\mathcal{Q}$ be the quotient:
$$
\mathcal{Q}=\left(\mathcal{Z}(\g \otimes \mathcal{S};V \otimes \mathcal{S}) \cap \mathcal{R}+\mathcal{B}(\g \otimes \mathcal{S};V \otimes \mathcal{S})\right)/\mathcal{B}(\g \otimes \mathcal{S};V \otimes \mathcal{S}).
$$
Now we shall state the main result of this work.

\begin{Theorem}\label{teorema de coeficientes universales}{\sl
Let $\g$ be a Lie algebra and $\mathcal{S}$ be a $m$-dimensional, associative and commutative algebra with unit. Let $\g \otimes \mathcal{S}$ be the current Lie algebra of $\g$ by $\mathcal{S}$. 
Let $\alpha:\mathcal{H}(\g \otimes \mathcal{S};V \otimes \mathcal{S})\to \mathcal{H}(\g;V) \otimes \mathcal{S}$ be the map of {\bf Prop. \ref{morfismo natural}}. Then the following short exact sequence is exact:
\begin{equation}\label{coeficientes universales}
\xymatrix@C+.42pc{
0 \ar[r] & \mathcal{Q} \ar[r]^>(.5){\subseteq} & \mathcal{H}(\g \otimes \mathcal{S};V \otimes \mathcal{S}) \ar[r]^<(.2){\alpha} & \mathcal{H}(\g;V) \otimes \mathcal{S} \ar[r] & 0
}
\end{equation}
}
\end{Theorem}
\begin{proof}
By \eqref{j1}, $\Theta$ is in $\mathcal{R}$ if and only if $\Theta_j(x_1 \otimes 1,\ldots,x_p \otimes 1)=\mathcal{L}(\Theta_j)(x_1,\ldots,x_p)=0$, for all $j$. Then $\Theta$ is in $\mathcal{R}$ if and only if $\Theta_j$ is in $\Ker(\mathcal{L})$, for all $j$.
\smallskip

In the proof of \textbf{Claim \ref{claim 5}}, we showed that $\Ker(\alpha)=\Ker(\Psi)$. We claim that $\Ker(\Psi)=\mathcal{Q}$. Let $\Lambda+\mathcal{B}(\g \otimes \mathcal{S};V \otimes \mathcal{S})$ be in $\Ker(\Psi)$. We will find an element $\Theta+\mathcal{B}(\g \otimes \mathcal{S};V \otimes \mathcal{S})$ in $\mathcal{Q}$ such that $\Lambda+\mathcal{B}(\g \otimes \mathcal{S};V \otimes \mathcal{S})=\Theta+\mathcal{B}(\g \otimes \mathcal{S};V \otimes \mathcal{S})$. By \eqref{psi} we have:
\begin{equation}\label{Thm0}
\Psi(\Lambda+\mathcal{B}(\g \otimes \mathcal{S};V \otimes  \mathcal{S}))\!=\!\sum_{j=1}^m(\mathcal{L}(\Lambda_j)+\mathcal{B})\! \otimes \!s_j=0.
\end{equation}
Then $\mathcal{L}(\Lambda_j)$ belongs to $\mathcal{B}$ for all $j$. Hence, there exists $\theta_j$ in $C^{p-1}(\g;V)$ such that $\mathcal{L}(\Lambda_j)=\d \theta_j$. 
\smallskip

Let $p>1$. For each $j$, let $\Delta_j$ be in $C^{p-1}(\g\otimes \mathcal{S};V)$ defined by:
$$
\Delta_j(x_1 \otimes t_1,\ldots,x_{p-1} \otimes t_{p-1})=\omega_1\left(t_1\cdots t_{p-1}\right)\theta_j(x_1,\ldots,x_{p-1}).
$$
Then $\mathcal{L}(\Delta_j)=\theta_j$ (see \eqref{j2}). Let $\Delta$ be in $C^{p-1}(\g \otimes \mathcal{S};V \otimes \mathcal{S})$ defined by:
\begin{equation*}\label{Thm1}
\Delta(\bar{x}_1,\ldots,\bar{x}_{p-1})\!=\!\Delta_1(\bar{x}_1,\ldots,\bar{x}_{p-1}) \otimes s_1\!+\!\ldots\!+\!\Delta_m(\bar{x}_1,\ldots,\bar{x}_{p-1}) \otimes s_m,
\end{equation*}
for all $\bar{x}_1,\ldots,\bar{x}_{p-1}$ in $\g \otimes \mathcal{S}$. From \eqref{j3} we have: 
$$
\mathcal{L}(\Lambda_j)=\d \theta_j=\d(\mathcal{L}(\Delta_j))=\mathcal{L}((\D\Delta)_j),\,\text{ for all }1 \leq j \leq m.
$$
Then there exists $\Theta_j$ in $\Ker(\mathcal{L})$ such that $\Lambda_j=(\D\Delta)_j+\Theta_j$. Let $\Theta$ be in $C^p(\g \otimes \mathcal{S};V \otimes \mathcal{S})$ defined by:
\begin{equation*}\label{Thm1-2}
\Theta(\bar{x}_1,\ldots,\bar{x}_{p})=\Theta_1(\bar{x}_1,\ldots,\bar{x}_{p})\otimes s_1+\ldots+\Theta_m(\bar{x}_1,\ldots,\bar{x}_{p})\otimes s_m,
\end{equation*}
for all $\bar{x}_1,\ldots,\bar{x}_p$ in $\g \otimes \mathcal{S}$. As $\Lambda_j=(\D\Delta)_j+\Theta_j$ for each $j$, then $\Lambda=\D \Delta+\Theta$ (see \eqref{widehat}). Due to $\Theta_j$ belongs to $\Ker(\mathcal{L})$, then $\Theta$ belongs to $\mathcal{R}$ and $\Lambda\!+\!\mathcal{B}(\g \otimes \mathcal{S};V \otimes \mathcal{S})\!=\!\Theta\!+\!\mathcal{B}(\g \otimes \mathcal{S};V \otimes \mathcal{S})$ belongs to $\mathcal{Q}$. 
\smallskip

Let $p=1$. By \eqref{Thm0}, $\mathcal{L}(\Lambda_j)$ belongs to $\mathcal{B}$, then there exists $v_j$ in $V$ such that $\mathcal{L}(\Lambda_j)=\d v_j $. Let $\Delta=v_1 \otimes s_1+\ldots+v_m \otimes s_m$. By \eqref{chi}, $\chi_j(\Delta)=v_j$. By definition of $\mathcal{L}$ (see \eqref{j2}), $\mathcal{L}(\chi_j(\Delta))=v_j$. Since $\mathcal{L}\circ \chi_j \circ \D=\d \circ \mathcal{L}\circ \chi_j$ (see \textbf{Prop. \ref{prop L y d conmutan}}) and $\chi_j(\D\Delta)=(\D \Delta)_j$ (see \eqref{widehat}), we obtain:
$$
\mathcal{L}(\Lambda_j)=\d v_j=\d(\mathcal{L}(\chi_j(\Delta)))=\mathcal{L}(\chi_j(\D \Delta))=\mathcal{L}((\D \Delta)_j).
$$
Then there exists $\Theta_j$ in $\Ker(\mathcal{L})$ such that $\Lambda_j=(\D\Delta)_j+\Theta_j$. Let $\Theta$ be in $C^1(\g \otimes \mathcal{S};V \otimes \mathcal{S})$ defined by:
$\Theta(\bar{x})=\Theta_1(\bar{x})\otimes s_1+\ldots+\Theta_m(\bar{x})\otimes s_m$, for all $\bar{x}$ in $\g \otimes \mathcal{S}$; then $\chi_j(\Theta)=\Theta_j$. As $\Lambda_j=(\D\Delta)_j+\Theta_j$ for each $j$, then $\Lambda=\D \Delta+\Theta$ (see \eqref{widehat}). Due to $\Theta_j$ belongs to $\Ker(\mathcal{L})$, then $\Theta$ belongs to $\mathcal{R}$. Thus for all $p\geq 1$, $\Lambda=\D \Delta+\Theta$.
\smallskip

Since $\D \Lambda=0$ then $\D \Theta=0$. Therefore $\Lambda+\mathcal{B}(\g \otimes \mathcal{S};V \otimes \mathcal{S})=\Theta+\mathcal{B}(\g \otimes \mathcal{S};V \otimes \mathcal{S})$ belongs to $\mathcal{Q}$, which proves that $\Ker(\alpha) \subset \mathcal{Q}$.
\smallskip

Now we assert that $\mathcal{Q} \subset \Ker(\alpha)$. Let $\Theta+\mathcal{B}(\g \otimes \mathcal{S};V \otimes \mathcal{S})$ be in $\mathcal{Q}$, where $\Theta$ is in $\mathcal{Z}(\g \otimes \mathcal{S};V \otimes \mathcal{S}) \cap \mathcal{R}$. As $\Theta(x_1 \otimes 1,\ldots,x_p \otimes 1)=0$, then $\Theta_j$ belongs to $\Ker(\mathcal{L})$ for all $j$. By \eqref{Thm0} we have that $\Theta+\mathcal{B}(\g \otimes \mathcal{S};V \otimes \mathcal{S})$ belongs to $\Ker(\Psi)=\Ker(\alpha)$. Then, $\Ker(\alpha)=\mathcal{Q}$.
\smallskip

For $p=0$, we decreed that $\mathcal{Q}^0=\{0\}$; in addition in \textbf{Claim \ref{T1}} we showed that $\mathcal{H}^0(\g \otimes \mathcal{S};V \otimes \mathcal{S})=\mathcal{H}^0(\g;V) \otimes \mathcal{S}$ and $\alpha$ is the identity.
\end{proof}

\subsection{Current Lie algebras over semisimple Lie algebras}

We will determine the cohomology group of $\g \otimes \mathcal{S}$, where $\g$ is a semisimple Lie algebra, and $V$ is an irreducible $\g$-module. We shall use that $\mathcal{H}(\g;V)=\{0\}$ (see \cite{Che}, \textbf{Thm. 24.1)}).

\begin{Prop}\label{semisimple}{\sl
Let $\g$ be a semisimple Lie algebra and $V$ an irreducible $\g$-module. Let $\mathcal{S}$ be an associative and commutative algebra with unit over a field $\F$. Then $\mathcal{H}(\g \otimes \mathcal{S};V \!\otimes \mathcal{S})\!=\!\mathcal{Q}$.}
\end{Prop}
\begin{proof}
As $\mathcal{H}(\g;V)=\{0\}$, by \textbf{Thm. \ref{teorema de coeficientes universales}} we get $\mathcal{H}(\g \otimes \mathcal{S};V \otimes \mathcal{S})=\mathcal{Q}$. Now we shall verify this result without using \textbf{Thm. \ref{teorema de coeficientes universales}}.
\smallskip

Let $\Lambda+\mathcal{B}(\g \otimes \mathcal{S};V \otimes \mathcal{S})$ be in $\mathcal{H}(\g \otimes \mathcal{S};V \otimes \mathcal{S})$. Then $\D \Lambda=0$, and $\widehat{\omega}_j\circ \D \Lambda=(\D \Lambda)_j=0$ for all $j$. By \textbf{Prop. \ref{prop L y d conmutan}}, $0=\mathcal{L}\left((\D \Lambda)_j\right)=\d \left(\mathcal{L}(\Lambda_j) \right)$, which implies that $\mathcal{L}(\Lambda_j)$ belongs to $\mathcal{Z}=\mathcal{B}$. Hence there exists $\mu_j$ in $C(\g;V)$ such that $\mathcal{L}(\Lambda_j)=\d \mu_j$; by \eqref{j2} we have:
\begin{equation}\label{semisimple 0}
\Lambda_j(x_1 \otimes 1,\ldots,x_p \otimes 1)=\d \mu_j(x_1,\ldots,x_p),\,\,\text{ for all }1 \leq j \leq m.
\end{equation}
Let $\Omega$ be in $C(\g \otimes \mathcal{S};V \otimes \mathcal{S})$ defined by:
\begin{equation}\label{definicion de mu}
\Omega(x_1 \otimes t_1,\ldots,x_p \otimes t_p)=\sum_{j=1}^m \mu_j(x_1,\ldots,x_p) \otimes s_j\,\bar{t},
\end{equation}
where $\bar{t}=t_1 \cdots t_p$. We claim that:
\begin{equation}\label{funcion auxiliar}
\D \Omega(x_1 \otimes t_1,\ldots,x_p \otimes t_p)=\sum_{j=1}^m \d \mu_j(x_1,\ldots,x_p) \otimes s_j\,\bar{t}.
\end{equation}
Indeed, if we write $\Omega$ as in \eqref{widehat} we obtain:
\begin{equation}\label{ss2}
\Omega(x_1 \otimes t_1,\ldots,x_p \otimes t_p)=\sum_{j=1}\Omega_j(x_1 \otimes t_1,\ldots,x_p \otimes t_p)\otimes s_j.
\end{equation}
From \eqref{definicion de mu}, \eqref{ss2} and the definition of $\mathcal{L}$, it follows: $\mathcal{L}(\Omega_j)(x_1,\ldots, x_p)=\Omega_j(x_1 \otimes 1,\ldots,x_p \otimes 1)\!=\!\mu_j(x_1,\ldots,x_p)$, then $\mathcal{L}(\Omega_j)\!=\!\mu_j$. By \eqref{j3} we get:
\begin{equation}\label{ss4}
\mathcal{L}\left((\D \Omega)_j\right)=\d \mathcal{L}(\Omega_j)=\d \mu_j\,\,\text{ for all }1 \leq j \leq m.
\end{equation}
By \textbf{Remark \eqref{remark}} and \eqref{definicion de mu}, $\mathcal{T}(\Id)(\Omega)=\Omega$. By \textbf{Prop. \ref{map of complexes}}, $\D \Omega=\D \mathcal{T}(\Id)(\Omega)=\mathcal{T}(\Id)(\D \Omega)$. Then by \eqref{j5} and \eqref{ss4} it follows:
$$
\aligned
& \D \Omega(x_1 \otimes t_1,\ldots,x_p \otimes t_p)=\mathcal{T}(\Id)(\D \Omega)((x_1 \otimes t_1,\ldots,x_p \otimes t_p))\\
&=\sum_{j=1}^m\mathcal{L}\left((\D \Omega)_j\right)(x_1,\ldots, x_p) \otimes s_j\,\bar{t}=\sum_{j=1}^m \d \mu_j(x_1,\ldots,x_p) \otimes s_j\, \bar{t},
\endaligned
$$
which proves \eqref{funcion auxiliar}. Let $\Theta=\Lambda-\D \Omega$. By \eqref{semisimple 0} and \eqref{funcion auxiliar} we obtain:
$$
\aligned
& \Theta(x_1 \otimes 1,\ldots,x_p \otimes 1)\!=\!\Lambda(x_1 \otimes 1,\ldots,x_p \otimes 1)-\D \Omega(x_1 \otimes 1,\ldots,x_p \otimes 1)\\
\,&=\sum_{j=1}^m \Lambda_j(x_1 \otimes 1,\ldots,x_p \otimes 1) \otimes s_j-\sum_{j=1}^m \d \mu_j(x_1,\ldots,x_p) \otimes s_j=0.
\endaligned
$$
Then $\Theta+\mathcal{B}(\g \otimes \mathcal{S};V \otimes \mathcal{S})$ belongs to $\mathcal{Q}$. Hence, $\mathcal{H}(\g \otimes \mathcal{S};V \otimes \mathcal{S})=\mathcal{Q}$.
\end{proof}

\section*{Appendix}

In this part of the work we prove the remaining cases of \textbf{Prop. \ref{composicion}}.

\begin{claim}{\sl
Let $f:C^p(\g;U) \to C^p(\g;V)$ and $g:C^p(\g;V) \to W$ be maps. Then $\mathcal{T}(g \circ f)=\mathcal{T}(g) \circ \mathcal{T}(f)$ is a map between $C^p(\g \otimes \mathcal{S};U \otimes \mathcal{S})$ and $W \otimes \mathcal{S}$.}
\end{claim}
\begin{proof}
Let $\Lambda$ be in $C^p(\g \otimes \mathcal{S};U \otimes \mathcal{S})$, then $\Theta=\mathcal{T}(\Lambda)(f)$ belongs to $C^p(\g \otimes \mathcal{S};V \otimes \mathcal{S})$. Using \eqref{j7} we get:
\begin{equation}\label{A1}
\mathcal{T}(g)(\Theta)=g \left(\mathcal{L}(\Theta_1)\right)\otimes s_1+\ldots+g \left(\mathcal{L}(\Theta_m)\right) \otimes s_m.
\end{equation}
By definition of $\mathcal{L}$ (see \eqref{j2}), for each $j$ we have:
\begin{equation}\label{A2}
\begin{split}
& \mathcal{L}(\Theta_j)(x_1,\ldots,x_p)\!=\!\Theta_j(x_1 \!\otimes \!1,\ldots,x_p\! \otimes \!1)\\
&=\!(\widehat{\omega}_j \circ \Theta)(x_1\! \otimes \!1,\ldots,x_p \!\otimes \!1)\!=\!\widehat{\omega}_j\left( \mathcal{T}(f)(\Lambda)(x_1\! \otimes \!1,\ldots,x_p \!\otimes \!1)\right),
\end{split}
\end{equation}
Using \eqref{j5} in $\mathcal{T}(f)(\Lambda)$ above we obtain:
\begin{equation}\label{A3}
\mathcal{T}(f)(\Lambda)(x_1 \otimes 1,\ldots,x_p \otimes 1)=\sum_{j=1}^m f\left(\mathcal{L}(\Lambda_j)\right)(x_1,\ldots,x_p)\otimes s_j.
\end{equation}
Applying $\widehat{\omega}$ in \eqref{A3} it follows:
\begin{equation}\label{A4}
\widehat{\omega}_j\left( \mathcal{T}(f)(\Lambda)(x_1 \otimes 1,\ldots,x_p \otimes 1)\right)=f\left(\mathcal{L}(\Lambda_j)\right)(x_1,\ldots,x_p),
\end{equation}
From \eqref{A2} and \eqref{A4} we deduce: $\mathcal{L}(\Theta_j)=f\left(\mathcal{L}(\Lambda_j)\right)$ for all $j$. Substitute this in \eqref{A1} we get:
\begin{equation}\label{A6}
\mathcal{T}(g)(\Theta)=\sum_{j=1}^m g\left(f\left(\mathcal{L}(\Lambda_j)\right)\right) \otimes s_j=\sum_{j=1}^m (g \circ f)\left(\mathcal{L}(\Lambda_j)\right) \otimes s_j.
\end{equation}
Due to $\Theta=\mathcal{T}(f)(\Lambda)$, by \eqref{A6} and \eqref{j7} it follows:
$$
\mathcal{T}(g) \circ \mathcal{T}(f)(\Lambda)=\sum_{j=1}^m (g \circ f)\left(\mathcal{L}(\Lambda_j)\right) \otimes s_j=\mathcal{T}(g \circ f)(\Lambda).
$$
Then $\mathcal{T}(g)\circ \mathcal{T}(f)=\mathcal{T}(g \circ f)$, which proves our claim.
\end{proof}

\begin{claim}{\sl
Let $f:C^p(\g;U) \to V$ and $g:V \to W$ be maps. Then $\mathcal{T}(g \circ f)=\mathcal{T}(g) \circ \mathcal{T}(f)$ is a map between $C^p(\g \otimes \mathcal{S};U \otimes \mathcal{S})$ and $W \otimes \mathcal{S}$.}
\end{claim}
\begin{proof}
Let $\Lambda$ be in $C^p(\g \otimes \mathcal{S};U \otimes \mathcal{S})$ and let $\bar{v}=\mathcal{T}(f)(\Lambda)$, which belongs to $V \otimes \mathcal{S}$. By \eqref{j7} we have:
\begin{equation}\label{A7}
\bar{v}=\mathcal{T}(f)(\Lambda)=f \left(\mathcal{L}(\Lambda_1)\right) \otimes s_1+\ldots+f \left(\mathcal{L}(\Lambda_m)\right) \otimes s_m.
\end{equation}
By definition, $\mathcal{T}(g)=g \otimes \mathcal{S}$. Applying $\mathcal{T}(g)$ to $\bar{v}$ in \eqref{A7} we obtain:
\begin{equation}\label{A8}
\mathcal{T}(g)(\bar{v})=(g \circ f)\left(\mathcal{L}(\Lambda_1)\right) \otimes s_1+\ldots+(g \circ f)\left(\mathcal{L}(\Lambda_m)\right) \otimes s_m.
\end{equation}
Using \eqref{j7} in \eqref{A8} we get
$$
\mathcal{T}(g)\circ \mathcal{T}(f)(\Lambda)=\sum_{j=1}^m(g \circ f)\left(\mathcal{L}(\Lambda_j)\right) \otimes s_j=\mathcal{T}(g \circ f)(\Lambda).
$$
From it follows $\mathcal{T}(g)\circ \mathcal{T}(f)=\mathcal{T}(g \circ f)$.

\end{proof}

\begin{claim}{\sl
Let $f:U \to V$ and $g:V \to C^p(\g;W)$ be maps. Then $\mathcal{T}(g \circ f)=\mathcal{T}(g)\circ \mathcal{T}(f)$ is a map between $U \otimes \mathcal{S}$ and $C^p(\g \otimes \mathcal{S};W \otimes \mathcal{S})$. 
}
\end{claim}
\begin{proof}
By definition we know that $\mathcal{T}(f)=f \otimes \mathcal{S}$. Let $u$ be in $U$, $s$ be in $\mathcal{S}$ and $\textbf{x} \otimes \textbf{t}=(x_1 \otimes t_1,\ldots,x_p \otimes t_p)$ (see \eqref{notacion}). By \eqref{j6} it follows:
\begin{equation*}
\begin{split}
& \mathcal{T}(g)\circ \mathcal{T}(f)(u \otimes s)(\textbf{x} \otimes \textbf{t})=\mathcal{T}(g)\left(f(u)\otimes s\right)(\textbf{x} \otimes \textbf{t})\\
&=g(f(u))(\textbf{x})\otimes s\,\bar{t}=\mathcal{T}(g \circ f)(u \otimes s)(\textbf{x} \otimes \textbf{t}),
\end{split}
\end{equation*}
where $\textbf{x}=(x_1,\ldots,x_p)$. Thus, $\mathcal{T}(g)\circ \mathcal{T}(f)=\mathcal{T}(g \circ f)$.
\end{proof}

\begin{claim}{\sl
Let $f:C^p(\g;U) \to V$ and $g:V \to C^p(\g;W)$ be maps. Then $\mathcal{T}(g)\circ \mathcal{T}(f)=\mathcal{T}(g \circ f)$ is a map between $C^p(g \otimes \mathcal{S};U \otimes \mathcal{S})$ and $C^p(\g \otimes \mathcal{S};W \otimes \mathcal{S})$.}
\end{claim}
\begin{proof}
Let $\Lambda$ be in $C^p(\g \otimes \mathcal{S};U \otimes \mathcal{S})$; by \eqref{j7} we have:
\begin{equation}\label{A9}
\mathcal{T}(f)(\Lambda)=f \left(\mathcal{L}(\Lambda_1)\right)\otimes s_1+\ldots+f \left(\mathcal{L}(\Lambda_m)\right)\otimes s_m.
\end{equation}
We apply $\mathcal{T}(g)$ to \eqref{A9}, and from \eqref{j6} we get:
\begin{equation}\label{A10}
\begin{split}
&\mathcal{T}(g)\circ \mathcal{T}(f)(\Lambda)(\textbf{x}\!\otimes \!\textbf{t})\!=\!\mathcal{T}(g)\!\left(\sum_{j=1}^m f\left(\mathcal{L}(\Lambda_j)\right)\!\otimes \!s_j\right)(\textbf{x}\!\otimes\! \textbf{t})\\
&=\!\sum_{j=1}^m g \left(f\left(\mathcal{L}(\Lambda_j)\right)\!\otimes \!s_j)\right)(\textbf{x})\!\otimes \!s_j\,\bar{t}=\!\sum_{j=1}^m (g \circ f)\left(\mathcal{L}(\Lambda_j)\right)(\textbf{x})\!\otimes \!s_j\,\bar{t}\\
&=\mathcal{T}(g \circ f)(\Lambda)(\textbf{x} \!\otimes \!\textbf{t}).
\end{split}
\end{equation}
In the last step we use \eqref{j5}. Then $\mathcal{T}(g)\circ \mathcal{T}(f)=\mathcal{T}(g \circ f)$.
\end{proof}

\section*{Acknowledgements}

The author thanks the support 
provided by post-doctoral fellowship
CONAHCYT 769309. The author has no conflicts to disclose.

\bibliographystyle{amsalpha}

\end{document}